\documentclass[12pt,reqno]{amsart}
\usepackage{enumerate, latexsym, amsmath, amsfonts, amssymb, amsthm, graphicx, color}
\def\pmod #1{\ ({\rm{mod}}\ #1)}
\def\Z{\Bbb Z}
\def\N{\Bbb N}

\def\Q{\Bbb Q}

\def\l{\left}
\def\r{\right}
\def\bg{\bigg}
\def\({\bg(}
\def\){\bg)}

\def\ord{{\rm ord}}

\def\gen{{\rm gen}}

\def\Aut{{\rm Aut}}
\def\ln{{\rm ln}}

\def\gs{\geqslant}

\def\ve{\varepsilon}

\def\eq{\equiv}

\def\da{\delta}

\theoremstyle{plain}
\newtheorem{theorem}{Theorem}

\newtheorem{lemma}{Lemma}
\newtheorem{corollary}{Corollary}
\newtheorem{conjecture}{Conjecture}

\theoremstyle{definition}

\theoremstyle{remark}
\newtheorem{remark}{Remark}

\makeatletter
\@namedef{subjclassname@2010}{%
  \textup{2010} Mathematics Subject Classification}
\makeatother
 \vspace{4mm}

\begin{document}
 \baselineskip=17pt
\hbox{}
\medskip
\title[Sums of squares with restrictions involving primes]
{Sums of squares with restrictions involving primes}
\date{}
\author[Hai-Liang Wu and Zhi-Wei Sun] {Hai-Liang Wu and Zhi-Wei Sun}

\thanks{
\newline \indent Supported by the National Natural Science
Foundation of China (Grant No. 11571162).}

\address {(Hai-Liang Wu)  Department of Mathematics, Nanjing
University, Nanjing 210093, People's Republic of China}
\email{\tt whl.math@smail.nju.edu.cn}

\address {(Zhi-Wei Sun) Department of Mathematics, Nanjing
University, Nanjing 210093, People's Republic of China}
\email{{\tt zwsun@nju.edu.cn}}

\begin{abstract}
 The well-known Lagrange's four-square theorem states that any integer $n\in\N=\{0,1,2,...\}$ can be written as the sum of four squares. Recently, Z.-W. Sun
 investigated the representations of $n$ as $x^2+y^2+z^2+w^2$ with certain linear restrictions involving the integer variables $x,y,z,w$. In this
 paper, via the theory of quadratic forms, we further study the representations $n=x^2+y^2+z^2+w^2$ (resp., $n=x^2+y^2+z^2+2w^2$) with certain linear restrictions involving primes. For example, we obtain the following results:

 (i) Each positive integer $n>1$ can be written as $x^2+y^2+z^2+2w^2$ ($x,y,z,w\in\mathbb N$) with $x+y$ prime.

 (ii) Every positive integer can be written as $x^2+y^2+z^2+2w^2$ ($x,y,z,w\in\mathbb N$) with $x+2y$ prime.

 (iii) Let $k$ be an arbitrary positive integer, and let $d$ be a positive odd integer with $4d^2+1$ prime.
Then any sufficiently large integer can be written as $x^2+y^2+z^2+2w^2$ $(x,y,z,w\in\mathbb N)$ with $x+2dy=p^k$ for some prime $p$.
\end{abstract}

\maketitle

\section{Introduction}
\setcounter{lemma}{0}
\setcounter{theorem}{0}
\setcounter{corollary}{0}
\setcounter{remark}{0}
\setcounter{equation}{0}
\setcounter{conjecture}{0}
\setcounter{proposition}{0}

In $1770$, Lagrange proved that any natural number $n$ can be written as the sum of four squares. This celebrated
result  now is well known as Lagrange's four-square theorem. The readers may consult \cite{G} for more details.
Motivated by this, many mathematicians have studied problems
involving sums of squares. In 1917, Ramanujan \cite{R} claimed that there are $55$ candidates of positive definite integral
diagonal quaternary
quadratic forms that represent all positive integers. Ten years later, Dickson \cite{D1} showed that
the quaternary form $x^2+2y^2+5z^2+5w^2$ included in Ramanujan's list represents all positive integers except $15$
and confirmed that Ramanujan's assertion for all the other $54$ forms is true.
Along this line, there are lots of research work on universal integral quadratic polynomials (a quadratic polynomial is said to be universal if it represents all natural numbers over $\Z$), see, e.g.,  the recent paper
of Sun \cite{S18}.

In contrast with the above, we may also consider the quadratic equation
$$f(x_1,x_2,...,x_k)=n.$$
If it is solvable over $\Z$ (or $\N$), then can we
find a particular solution $(x_1,x_2,...,x_k)\in\Z^k$ (resp. $\N^k$) satisfying some given algebraic conditions?
A typical example of this question is the following Cauchy's Lemma (cf. \cite[p.\,31]{N}).
\medskip

\noindent{\bf Cauchy's Lemma}. {\it  Let $n$ and $m$ be positive odd integers with $m^2<4n$ and $3n<m^2+2m+4$. Then
there exist $s,t,u,v\in\N$ such that
$$n=s^2+t^2+u^2+v^2\ \text{and}\ s+t+u+v=m.$$}
Let $P(x,y,z,w)$ be a polynomial in $\Z[x,y,z,w]$ and let $\mathcal{S}$ be a given subset
of $\N$ (e.g., the set of squares, or the set of primes).
The second author \cite{S17a} investigated the equation
\begin{equation*}
\left\{
\begin{aligned}
x^2+y^2+z^2+w^2&=n,\\
P(x,y,z,w)\in\mathcal{S}.\\
\end{aligned}
\right.
\end{equation*}
For example, he proved that for each natural number $n$
there exist $x,y,z,w\in\N$ such that $n=x^2+y^2+z^2+w^2$ and $x^4+8y^3z+8yz^3$ is a square. Moreover,
he \cite{S17a} would like to offer \$1350 (USD) for the first solution of his following conjecture.
\medskip

\noindent{\bf 1-3-5 Conjecture} (Sun \cite{S17a}). {\it Any $n\in\N$ can be written as $x^2+y^2+z^2+w^2$
$(x,y,z,w\in\N)$ with $x+3y+5z$ a square.}
\medskip

This conjecture has been verified for $n$ up to $10^{10}$ by Q.-H. Hou at
Tianjin Univ. Here we list some progress on this conjecture.
\begin{itemize}
\item
Y.-C. Sun and Z.-W. Sun \cite[Theorem 1.8]{SS} used Euler's four-square
identity to show that any $n\in\N$ can be written as
$x^2+y^2+z^2+w^2$ with $x,y,5z,5w\in\Z$ such that $x+3y+5z$ is a square.

\item
In \cite[Theorem 1.2]{WS}, with the help of half-integral modular forms, we proved that
there is a finite set $A$ of positive integers such that any sufficiently large integer not in the
set $\{16^ka:\ a\in A,\ k\in\N\}$ can be written as
$x^2+y^2+z^2+w^2$ with $x,y,z,w\in\Z$ and $x+3y+5z\in\{4^k:\ k\in\N\}$.

\end{itemize}
In this paper, we focus on the representations $n=x^2+y^2+z^2+w^2$ (or $n=x^2+y^2+z^2+2w^2$) with
some additional linear restrictions on the variables $x,y,z,w$.

Before stating our results, we first recall the following known fact:
Each number in the set $\bigcup_{r\in\N}\{4^r\times2, 4^r\times6, 4^r\times14\}$ has only a single partition
into four squares (cf. \cite[p. 86]{G}), namely,
\begin{align*}
4^r\times2=&(2^r)^2+(2^r)^2+0+0,
\\4^r\times6=&(2^{r+1})^2+(2^r)^2+(2^r)^2+0,
\\4^r\times14=&(3\times2^r)^2+(2^{r+1})^2+(2^r)^2+0.
\end{align*}
On the other hand, any integer
of the form $4^r(8l+7)\ (k,l\in\N)$ can not be written as the sum of three squares. Therefore, for any $a,b,c,d\in\N$
with $a^2+b^2+c^2+d^2\ne0$, there does not exist a finite subset $\mathcal{S}$ of $\N$
such that any sufficiently large integer $n$ can be written as $x^2+y^2+z^2+w^2\ (x,y,z,w\in\N)$
with $ax+by+cz+dw\in \mathcal{S}.$

In view of this fact, we turn to study the representations of $n\in\N$ by
$x^2+y^2+z^2+2w^2$ with certain additional linear restriction.

\begin{theorem}\label{Thm 1.1}
Let $k$ be a positive integer, and let $d$ be a positive odd integer with $1+4d^2$ prime.
Then any sufficiently large integer $n$ can be written as $x^2+y^2+z^2+2w^2$ $(x,y,z,w\in\N)$ with $x+2dy=p^k$ for
some prime $p$.
\end{theorem}

\begin{corollary}\label{Cor 1.1}
Let $d$ be a positive odd integer with $4d^2+1$ prime.
Then any sufficiently large integer $n$ can be written as $x^2+y^2+z^2+2w^2$ $(x,y,z,w\in\N)$ with $x+2dy$ prime.
\end{corollary}

\begin{remark}
For some particular $d$ satisfying the above condition, we can give an effective bound $\mathcal{B}(d)$ such that
each integer $n>\mathcal{B}(d)$ can be written as $x^2+y^2+z^2+2w^2$ $(x,y,z,w\in\N)$ with $x+2dy$ prime.
In particular, when $d=1$ we obtain the following result for all positive integers.
\end{remark}

\begin{corollary}\label{Cor 1.2}
For any positive integer $n$, there exist $x,y,z,w\in\N$ and a prime $p$ such that
$$n=x^2+y^2+z^2+2w^2\ \text{and}\ x+2y=p.$$
\end{corollary}
\begin{remark} In contrast, Sun \cite[Conjecture 4.3(ii)]{S17b} conjectured that any integer $n>1$ not divisible by $4$ can be written as $x^2+y^2+z^2+w^2$ ($x,y,z,w\in\N$) such that $p=x+2y+5z$, $p-2$, $p+4$ and $p+10$ are all prime.
\end{remark}

Next we consider the restrictions $x+y$ and $x+w$.
\begin{corollary}\label{Cor 1.3}
{\rm (i)} For any positive integer $n>1$, there are $x,y,z,w\in\N$ and a prime $p$ such that
$$n=x^2+y^2+z^2+2w^2\ \text{and}\ x+y=p.$$

{\rm (ii)} Any positive integer $n>2$ can be written as $x^2+y^2+z^2+2w^2$ with $x+w$ prime.
\end{corollary}

Now we state our second theorem.

\begin{theorem}\label{Thm 1.2}
Let $\lambda$ be an arbitrary positive odd integer.

{\rm (i)} Any integer $n\ge \lambda^2/8$ can be written as
$x^2+y^2+z^2+2w^2$ $(x,y,z,w\in\Z)$ with $x+y+2z+2w=\lambda.$

{\rm (ii)} Each integer $n\ge \lambda^2/16$ can be written as
$x^2+y^2+z^2+2w^2$ $(x,y,z,w\in\Z)$ with $x+2y+3z+2w=\lambda.$

{\rm (iii)} If $7\nmid \lambda$, then for each odd integer $n\ge\lambda/\sqrt{14}$, there exist
$x,y,z,w\in\Z$ such that
$$n^2=x^2+y^2+z^2+w^2\ \text{and}\ x+2y+3z=\lambda.$$
\end{theorem}

Part (ii) of Theorem \ref{Thm 1.2} implies that any positive integer can be written as
$x^2+y^2+z^2+2w^2$ $(x,y,z,w\in\Z)$ with $x+2y+3z+2w=1$, which appeared in Sun \cite[Conjecture 4.13(i)]{S17b}. Part (iii) of Theorem \ref{Thm 1.2} is motivated by a conjecture of Sun
\cite[Conjecture 4.15(ii)]{S17b} which states that any positive square can be written as
$x^2+y^2+z^2+w^2\ (x,y,z,w\in\Z)$ with $x+2y+3z$ a power of $4$.

The second author \cite{S17b} obtained the following identity:
\begin{align*}
7(s^2+t^2+u^2+2w^2)=x^2+y^2+z^2+2w^2,
\end{align*}
where $x=s+2u+2v,\ y=-2t-u+2v,\ z=2s-t-2v,\ w=s+t-u+v$. With the help of this identity, he proved that
any $n\in\N$ can be written as $x^2+y^2+z^2+2w^2$ $(x,y,z,w\in\Z)$ with $x+y+z+w$ a square (or a cube). In this
paper, using a new method, we prove a stronger result.

\begin{theorem}\label{Thm 1.3} Let $\lambda$ be an arbitrary
integer not divisible by $7$. Then, for any $\da\in\{0,1\}$ and any integer $n\gs\lambda^2/7$
with $n\not\eq\lambda\pmod{2^{1+\da}}$,
 there are $x,y,z,w\in\Z$ such that
$$2n+\da=x^2+y^2+z^2+2w^2\ \text{and}\ x+y+z+w=\lambda.$$
\end{theorem}

To state more results, we need to introduce some notations.
For any positive integers $j,k,l$ we set
\begin{equation}\label{bound A}
a_k(j,l)=(2kl)^{2k}(j+1)^{2k-1},
\end{equation}
\begin{equation}\label{bound B}
b_k(j,l)=\frac{1}{j}\exp\l(2k^{3/2}j^{1/4k}(j+1)^{(2k-1)/4k}\r),
\end{equation}
\begin{equation}\label{bound C}
c_k(j,l)=\max\{a_k(j,l), b_k(j,l), (3275)^{2k}/j\}.
\end{equation}

\begin{theorem}\label{Thm 1.4}
Let $k$ be a positive integer.

{\rm (i)} Any odd integer $n>c_k(3,4)$
can be written as $x^2+y^2+z^2+w^2$ $(x,y,z,w\in\N)$ with $x+y+z+w=p^k$ for some prime $p$. Also, any even integer
$n>a_k(3,2)$ with $n\equiv 2\pmod4$ can be written as
$x^2+y^2+z^2+w^2$ $(x,y,z,w\in\N)$ with $x+y+z+w=(2b)^k$ for some $b\in\N$.

{\rm (ii)} Any each odd integer $n>c_k(5,6)$ can be written as
$x^2+y^2+2z^2+2w^2$ $(x,y,z,w\in\N)$ with $x+y+2z+2w=p^k$ for some prime $p$. Also, each even integer
$n>a_k(5,6)$ with $4\mid n$ can be written as $x^2+y^2+2z^2+2w^2$ $(x,y,z,w\in\N)$ with $x+y+2z+2w=(2b)^k$
for some $b\in\N$.

{\rm (iii)} Each odd integer $n>c_k(5,4)$ can be written as $x^2+y^2+z^2+3w^2$ $(x,y,z,w\in\N)$ with
$x+y+z+3w=p^k$ for some prime $p$. Also, any even integer $n>a_k(5,4)$ can be written as
$x^2+y^2+z^2+3w^2$ $(x,y,z,w\in\N)$ with $x+y+z+3w=(2b)^k$ for some $b\in\N$.
\end{theorem}

\maketitle
\section{Proofs of Theorem \ref{Thm 1.1} and corollaries \ref{Cor 1.2}-\ref{Cor 1.3}}
\setcounter{lemma}{0}
\setcounter{theorem}{0}
\setcounter{corollary}{0}
\setcounter{remark}{0}
\setcounter{equation}{0}
\setcounter{conjecture}{0}

In this section, we will use the theory of quadratic forms to prove Theorem \ref{Thm 1.1} and Corollaries \ref{Cor
1.2}-\ref{Cor 1.3}.
For convenience, we will adopt the language of both classical quadratic forms and lattice theory. Any unexplained
notations can be found in \cite{C,Ki,Oto}.

Let $\Omega$ be the set of all places of $\Q$ and $J_{\Q}$ be the set
of ideles of $\Q$. Moreover, let $\theta$ denote the spinor norm map. For a positive definite lattice $L$,
set
$$J_{L}=\{i=(i_v)\in J_{\Q}:\ i_v\in\theta(O^+(L_v))\ \text{for all}\ v\in\Omega\}.$$
We first need the following well known result involving the number of  proper spinor genera in a given genus
(cf. \cite[Theorem 6.3.1]{Ki}).

\begin{lemma}\label{number of spinor genus}
Let notations be as the above. If rank$(L)\ge3$, then the number of proper spinor genera in $\gen(L)$
is equal to the index $[J_{\Q}:\Q^{\times}J_{L}]$.
\end{lemma}

\begin{remark}When $rank(L)=3$, it is easy to see that the number of spinor genera in $\gen(L)$ is equal to
the number of proper spinor genera in $\gen(L)$.
\end{remark}

On the other hand, we also need some analytic theory of quadratic forms.
Let $f$ be a positive definite integral ternary quadratic form. Suppose that there is only a single spinor genus in
$\gen(f)$. Let $\Aut(f)$ denote the group of integral isometries of $f$.
 For $n\in\N$, set
 $$r(f,n):=\#\{(x,y,z)\in \Z^3 : f(x,y,z)=n\}$$
 (where $\#S$ denotes the cardinality of a finite set $S$), and let
  \begin{equation*}r(\gen(f),n):=\bigg(\sum_{f^*\in \gen(f)}\frac{1}{\# \Aut(f^*)}\bigg)^{-1}\sum_{f^*\in
  \gen(f)}\frac{r(f^*,n)}{\# \Aut(f^*)},
  \end{equation*}
  where the summation is over a set of representatives of the classes in $\gen(f)$.
It is well known that the theta series
\begin{equation*}
\theta_f(z)=\sum_{n\gs0}r(f,n)e^{2\pi inz}
\end{equation*}
(with $z$ in the upper half plane) is a modular form of weight $3/2$. According to W.
Duke and R. Schulze-Pillot's outstanding work \cite{WR}, if we write
\begin{align*}
\theta_f&=\theta_{\gen(f)}+(\theta_f-\theta_{\gen(f)}),
\end{align*}
 then
 $\theta_{\gen(f)}$
 is an Eisenstein series, and $\theta_f-\theta_{\gen(f)}$ is a cusp form whose Shimura lift is also a cusp form. By the
 work of W. Duke \cite{Duke}, for any $\ve>0$, the $n$-th Fourier coefficient of the cusp form
 $\theta_f-\theta_{\gen(f)}$ is (ineffectively) $\ll
n^{1/2-1/28+\ve}$. On the other hand, if $n\in\N$ is represented by $\gen(f)$ and $n$ has bounded divisibility at
each anisotropic prime, then by \cite[Lemma 5]{WR}, $r(\gen(f),n) \gg n^{1/2-\ve}$. This implies the
following result.

\begin{lemma}\label{modular forms}
Let $f$ be a positive definite integral ternary quadratic form. Suppose that there is only a single spinor genus in
$\gen(f)$. For an integer $n$, if $n$ can be represented by $\gen(f)$ and $n$ has bounded divisibility at each
anisotropic prime, then $n$ can be represented by $f$ provided that $n$ is sufficiently large.
\end{lemma}

We also need the following result on the distribution of primes in intervals.
\begin{lemma}\label{prime}{\rm (P. Dusart \cite{DUS})}
For $x>3275$, the interval $[x,x+\frac{x}{2\ln^2x}]$ contains at least one prime.
\end{lemma}

\noindent{\it Proof of Theorem \ref{Thm 1.1}}. Set $1+4d^2=q$ with $q\equiv5\pmod8$, and set
$f(x,y,z)=x^2+qy^2+2qz^2$. For any prime $r$ not dividing $2q$, $f(x,y,z)$ is unimodular in $\Z_r$ (where $\Z_r$
denotes the set of $r$-adic integers) and hence $f$ can represent
all $r$-adic integers over $\Z_r$. Thus
$\Z_r^{\times}\Q_r^{\times2}\subseteq \theta(O^+(f_r))$. For prime $q$, by scaling, it is easy to see that
$\Z_q^{\times}\Q_p^{\times2}\subseteq \theta(O^+(f_q))$. For prime $2$, since $q\equiv5\pmod8$, $f$
is $\Z_2$-equivalent to the form $x^2+5y^2+10z^2$. By \cite[pp. 112--113]{D2} each integer of the form $5n+1$ can be
represented by $x^2+5y^2+10z^2$ over $\Z_2$. Hence $f$ can represent all $2$-adic integers over $\Z_2$. So we
have
$\Z_2^{\times}\Q_2^{\times2}\subseteq \theta(O^+(f_2))$. On the other hand, for each $i=(i_v)\in J_{\Q}$,
there is a $y\in\Q^{\times}$ such that $y\cdot i_v\in\Z_v^{\times}$ for each finite place $v$ and $y\cdot
i_{\infty}>0$.
By Lemma \ref{number of spinor genus} it is easy to see there is a single spinor genus in $\gen(f)$.

Let $\mathcal{I}=[(4d^2n)^{1/2k},((1+4d^2)n)^{1/2k}]$. If $n$ is large enough, we may assume that
the length of the interval $\mathcal{I}$ is greater than
$\frac{(4d^2n)^{1/2k}}{2\ln^2((4d^2n)^{1/2k})}$ and that $(4d^2n)^{1/2k}>{\rm max}\{1+4d^2,3275\}$. Thus by Lemma \ref{prime} there is a prime $p\in\mathcal{I}$ with $p>q=1+4d^2$. Now we consider the number
$c(n,p):=(1+4d^2)n-p^{2k}$. By the above discussion we see that $f(x,y,z)$ can represent all $r$-adic integers
over $\Z_r$ whenever $r$ is a prime not equal to $q$. When $r=q$, note that $q\equiv1\pmod4$ and $q\nmid p$, thus
$-p^2$ is a quadratic residue modulo $q$. By Hensel's Lemma it is easy to see that $f(x,y,z)$ can represent $c(n,p)$
over
$\Z_q$. Moreover, by the above assumption we have $c(n,p)\ge0$.

In view of the above, $c(n,p)$ can be represented
by $\gen(f)$ and $c(n,p)$ has bounded divisibility at $q$ (it is easy to see that the only anisotropic prime of $f$
is $q$).
Hence, when $n$ is sufficiently large, by Lemma \ref{modular forms} and the fact that $(-2dp^k)^2\equiv -p^{2k}\pmod
q$,
there exist
$z,w\in\N$ and $s\in\Z$ with $s\equiv -2dp^k\pmod q$ such that $qn-p^{2k}=s^2+qz^2+2qw^2$. Clearly,
$-\sqrt{qn-p^{2k}}\le s\le \sqrt{qn-p^{2k}}$. We therefore have
\begin{equation*}
s+2dp^k\ge -\sqrt{qn-p^{2k}}+2dp^k=(1+4d^2)(p^{2k}-n)/(\sqrt{qn-p^{2k}}+2dp^k)\ge0.
\end{equation*}
Hence we may write $s+2dp^k=(1+4d^2)y$ with $y\ge0$. Then
\begin{equation*}
(1+4d^2)n-p^{2k}=((1+4d^2)y-2dp^k)^2+(1+4d^2)z^2+2(1+4d^2)w^2
\end{equation*}
with $y,z,w\in\N$. This implies that
\begin{equation}\label{equation A}
n=(p^k-2dy)^2+y^2+z^2+2w^2.
\end{equation}
Furthermore, we have
\begin{align*}
(1+4d^2)(p^k-2dy)&=p^k-2ds\ge p^k-2d\sqrt{(1+4d^2)n-p^{2k}}\\
&=\frac{(1+4d^2)(p^{2k}-4d^2n)}{p^k+2d\sqrt{(1+4d^2)n-p^{2k}}}\ge0.
\end{align*}
This implies that $x:=p^k-2dy\ge0$. By (\ref{equation A}) there exist $x,y,z,w\in\N$ and a prime $p$ such that
$n=x^2+y^2+z^2+2w^2$ with $x+2dy=p^k$. This completes the proof.\qed

\medskip
\noindent{\it Proof of Corollary \ref{Cor 1.2}}. Let the notations be as in the proof of Theorem \ref{Thm 1.1}. When $d=1$, we have $f(x,y,z)=x^2+5y^2+10z^2$ and $\mathcal{I}=[\sqrt{4n},\sqrt{5n}]$.
When $1\le n\le 3275^2/4$, via computation, it is easy to verify our result. Suppose now that $n>3275^2/4$. By Lemma \ref{prime} there is a prime $p>5$ in $\mathcal{I}$.

By \cite[pp. 112--113]{D2}, $f(x,y,z)$ is a regular form which can represent $5n-p^2$. Hence $5n-p^2=f(s,z,w)$ for some $z,w\in\N$ and $s\in\Z$.
Then by the essentially same method in the proof of Theorem \ref{Thm 1.1}, we can obtain the desired result.\qed

\noindent{\it Proof of Corollary \ref{Cor 1.3}}. (1) Let $g(x,y,z)=x^2+2y^2+4z^2$, and let $\mathcal{I}_1=[\sqrt{n},\sqrt{2n}]$. When $n\in\{2,3,\cdots,3275^2\}$, via computer, it is easy to verify our results. Suppose now $n>3275^2$. When this
occurs, by easy computation and Lemma \ref{prime}, the interval $\mathcal{I}_1$ contains at least one odd prime $p$.

By \cite[pp. 112--113]{D2}, $g(x,y,z)$ is a regular form which can represent $2n-p^2$. Hence $2n-p^2=g(s,z,w)$ for some
$z,w\in\N$ and $s\in\Z$. Clearly $s\equiv -p\pmod2$ and
$$s+p\ge-\sqrt{2n-p^2}+p=\frac{2(p^2-n)}{\sqrt{2n-p^2}+p}\ge0.$$
Hence we can write $s=2y-p$ with $y\in\N$. And we have
$$2n-p^2=(2y-p)^2+2z^2+4w^2.$$
This implies that
\begin{equation}\label{equation H}
n=(p-y)^2+y^2+z^2+2w^2.
\end{equation}
Moreover, note that
$$x:=p-y=\frac{p-s}{2}\ge\frac12(p-\sqrt{2n-p^2})=\frac{p^2-n}{p+\sqrt{2n-p^2}}\ge0.$$
Then the desired result follows from (\ref{equation H}).

(2) Let $h(x,y,z)=x^2+3y^2+3z^2$, and let $\mathcal{I}_2=[\sqrt{n},\sqrt{3n/2}]$. When $n\in\{3,\cdots 3275^2\}$, with the help of computer, it is easy to verify the result. Suppose now $n>3275^2$. By computation and Lemma \ref{prime}, we see that the interval $\mathcal{I}_2$ contains at least one prime $p>3$.

By \cite[pp. 112--113]{D2}, $h(x,y,z)$ is a regular form which can represent $3n-2p^2$. Thus $3n-2p^2=h(s,y,z)$ for
some $s\in\Z$ and $y,z\in\N$. Changing the sign of $s$ if necessary, we assume that $s\equiv -p\pmod3$. Since
$$(s+p)\ge(-\sqrt{3n-2p^2}+p)=\frac{3(p^2-n)}{\sqrt{3n-2p^2}+p}\ge0,$$
we may write $s=3w-p$ with $w\in\N$. And we have
$$3n-2p^2=(3w-p)^2+3y^2+3z^2.$$
This gives
\begin{equation}\label{equation I}
n=(p-w)^2+y^2+z^2+2w^2.
\end{equation}
In addition, note that
$$x:=p-w=(2p-s)/3\ge(2p-\sqrt{3n-2p^2})/3=\frac{2p^2-n}{2p+\sqrt{3n-2p^2}}\ge0.$$
Clearly the desired result follows from (\ref{equation I})\qed

\maketitle
\section{Proofs of Theorems \ref{Thm 1.2}-\ref{Thm 1.3}}
\setcounter{lemma}{0}
\setcounter{theorem}{0}
\setcounter{corollary}{0}
\setcounter{remark}{0}
\setcounter{equation}{0}
\setcounter{conjecture}{0}

\begin{lemma}\label{Lem A}
Given a positive odd integer $\lambda$, for each integer $n\ge \lambda^2/16$, $80n-5\lambda^2$ can be written as
$a^2+10b^2+16c^2$ $(a,b,c\in\Z)$ with $a\equiv -\lambda\pmod8$, $b\equiv \lambda\pmod4$, $c+2a\equiv0\pmod5$ and
$a+2b\equiv 9\lambda\pmod{16}$.
\end{lemma}
\begin{proof}
By \cite[pp. 112--113]{D2} we can write $16n-\lambda^2=x_1^2+x_2^2+2x_3^2$ with $x_1,x_2,x_3\in\Z$ and $2\nmid x_1$.
Since $16n-\lambda^2\equiv 3\equiv 1+2x_3^2\pmod4$, we obtain that $x_3$ is odd. Moreover, as
$16n-\lambda^2\equiv -1\equiv 1+x_2^2+2\pmod8$, we have $x_2\equiv 2\pmod4$. Note that the identity:
$$80n-5\lambda^2=5(16n-\lambda^2)=(x_1-2x_2)^2+(2x_1+x_2)^2+10x_3^2.$$
Set $a_1=x_1-2x_2,\ b_1=x_3,\ c_1=(2x_1+x_2)/4$. By the above discussion, it is easy to see that $a_1,b_1,c_1\in\Z$.
Then we have $80n-5\lambda^2=a_1^2+10b_1^2+16c_1^2$ with $2\nmid a_1b_1$.

Since $b_1$ is odd, without loss of generality, we may assume that $b_1\equiv \lambda\pmod4$ (otherwise we replace $b_1$ by $-b_1$). Moreover, as $80n-5\lambda^2\equiv -5\lambda^2\equiv a_1^2+10\pmod{16}$, we have
$a_1\equiv \pm\lambda\pmod8$. Without loss of generality, we may assume that $a_1\equiv -\lambda\pmod8$ (otherwise we
replace $a_1$ by $-a_1$). On the other hand, since $80n-5\lambda^2\equiv 0\equiv a_1^2-4c_1^2\pmod5$, by changing the
sign of $c_1$ if necessary, we assume that $c_1+2a_1\equiv0\pmod5$.

Finally, as $a_1\equiv -\lambda\pmod 8$ and
$b_1\equiv \lambda\pmod4$, we have $a_1+2b_1\equiv \lambda\pmod 8$. If $a_1+2b_1\equiv 9\lambda\pmod{16}$, let
$a=a_1,\ b=b_1,\ c=c_1$, then we get the desired result. Suppose now that $a_1+2b_1\equiv \lambda\pmod {16}$. Let
$g(x,y,z)=x^2+10y^2+16z^2$. We have
\begin{equation}\label{equation B}
g\l(\frac{-3x+16z}{5},\ y,\ \frac{x+3z}{5}\r)=g(x,y,z).
\end{equation}
By (\ref{equation B}), we set
$a=(-3a_1+16c_1)/5,\ b=b_1,\ c=(a_1+3c_1)/5$. By the above congruence relationships between $a_1,b_1,c_1$, it is easy
to see that $a,b,c$ are all integers. Moreover, $a\equiv -3a_1/5\equiv a_1\equiv -\lambda\pmod 8$, $b=b_1\equiv
\lambda\pmod 4$, $c+2a=-a_1+7c_1\equiv0\pmod5$ and $a+2b=(-3a_1+16c_1+10b_1)/5\equiv (-3a_1-6b_1)/5\equiv
9\lambda\pmod{16}$.

In view of the above, we complete the proof.
\end{proof}

\begin{lemma}\label{Lem B}
Let $\lambda$ be a positive odd integer. for each integer $n\ge \lambda^2/8$, we can write
$88n-11\lambda^2=a^2+8b^2+44c^2$ $(a,b,c\in\Z)$ with $a+5b\equiv0\pmod{11}$, $2\nmid c$ and $a\equiv \lambda\pmod8$.
\end{lemma}
\begin{proof}
By \cite[pp. 112--113]{D2}, we can write $8n-\lambda^2=x_1^2+x_2^2+2x_3^2$ with $x_1,x_2,x_3\in\Z$ and $2\nmid x_1$.
Since $8n-\lambda^2\equiv -1\equiv 1+2x_3^2\pmod4$, then $x_3$ is odd. Moreover, as
$8n-\lambda^2\equiv-1\equiv 1+x_2^2+2\pmod8$, we have $x_2\equiv2\pmod4$.

By changing the sign of $x_1$ if necessary, without loss of generality, we may assume that $x_1\equiv\lambda\pmod4$.
At this time, $(3x_1-\lambda)/2$ is an odd integer. Thus by changing the sign of $x_3$ if necessary, we may assume
that
$x_3\equiv(3x_1-\lambda)/2\pmod4$. Note that the identity:
\begin{align*}
88n-11\lambda^2&=11(x_1^2+x_2^2+2x_3^2)\\
&=(3x_1-2x_3)^2+2(x_1+3x_3)^2+11x_2^2.
\end{align*}
Set $a=3x_1-2x_3,\ b_1=(x_1+3x_3)/2,\ c=x_2/2$, then we have $88n-11\lambda^2=a^2+8b_1^2+44c^2$ and
$a\equiv 3x_1-2x_3\equiv3x_1-(3x_1-\lambda)\equiv \lambda\pmod8$. Since $88n-11\lambda^2\equiv0\equiv
(a+5b_1)(a-5b_1)\pmod{11}$, there is an $\ve\in\{\pm1\}$ such that $a+5\ve b\equiv0\pmod{11}$. Let $b=\ve b_1$.
Then $88n-11\lambda^2=a^2+8b^2+44c^2$ with $a,b,c$ satisfying the desired conditions.

This completes the proof.
\end{proof}
\noindent{\it Proof of Theorem \ref{Thm 1.2}}. {\rm (i)} Let $\lambda$ be a fixed positive odd integer. When
$n\ge\lambda^2/8$, by Lemma \ref{Lem B} we can write $88n-11\lambda^2=a^2+8b^2+44c^2$ $(a,b,c\in\Z)$ with $a,b,c$ satisfying the congruence conditions
described
in Lemma \ref{Lem B}. Thus there exists an $r\in\{0,1,...,10\}$ such that $a\equiv \lambda-16r\pmod8$ and $a\equiv
\lambda-16r\pmod{11}$.
By the Chinese Remainder Theorem, we may write $a=88s+\lambda-16r$ with $s\in\Z$. Since $a+5b\equiv0\pmod{11}$, there
exists an integer $t$ such that $b=11t+2\lambda+r$. Moreover, as $c$ is odd, we may set $c=2u-\lambda$ with $u\in\Z$.
Hence we have the identity:
\begin{equation*}
88n-11\lambda^2=(88s+\lambda-16r)^2+8(11t+2\lambda+r)^2+44(2u-\lambda)^2.
\end{equation*}
This implies that
\begin{equation}\label{equation D}
n=(\lambda+s+2t-u)^2+(s+2t+u)^2+(r-6s-t)^2+2(-r+5s-t)^2.
\end{equation}
Let $x=\lambda+s+2t-u,\ y=s+2t+u,\ z=r-6s-t,\ w=-r+5s-t$. By (\ref{equation D}) we have $n=x^2+y^2+z^2+2w^2$ and
$x+y+2z+2w=\lambda$. This completes the proof of {\rm (i)} of Theorem \ref{Thm 1.2}.

{\rm (ii)} Given a positive odd integer $\lambda$, when
$n>\lambda^2/16$, by Lemma \ref{Lem A}, there exist integers
$a,b,c$ satisfying the congruence conditions described in
Lemma \ref{Lem A} such that $80n-5\lambda^2=a^2+10b^2+16c^2$.

 Thus there exists an $r\in\{0,1,2,...,9\}$ such that $a\equiv 7\lambda+8r\pmod{16}$ and $a\equiv 7\lambda+8r\pmod5$.
 Then by the Chinese Remainder Theorem, we may set $a=80s+7\lambda+8r$ with $s\in\Z$. Since
 $c+2a\equiv0\pmod5$, there exist a $u\in\Z$ such that $c=5u+\lambda-r$. On the other hand, as
 $a+2b\equiv 9\lambda\pmod {16}$, we may set $b=8t+\lambda-4r$ with $t\in\Z$. Hence we have the identity:
 \begin{equation*}
 80n-5\lambda^2=(80s+7\lambda+8r)^2+10(8t+\lambda-4r)^2+16(5u+\lambda-r)^2.
 \end{equation*}
This implies that
\begin{equation}\label{equation C}
n=(\lambda+7s+t+u)^2+(-r-2s+2t-u)^2+(-3s-t+u)^2+2(r+3s-t-u)^2.
\end{equation}
Let $x=\lambda+7s+t+u,\ y=-r-2s+2t-u,\ z=-3s-t+u,\ w=r+3s-t-u$. Then the desired result follows from (\ref{equation C}).

{\rm (iii)} Let $R(x,y,z)=3x^2+5y^2+14z^2+2xy$. Since the discriminant of $R$ is $196$, when prime $p\not\in\{2,7\}$, $R(x,y,z)$ is unimodular in $\Z_p$ and hence $R(x,y,z)$ can represent all
$p$-adic integers over $\Z_p$. When $p=7$, $R(x,y,z)$ is $\Z_7$-equivalent to $3x^2-7y^2+14z^2$. Noting that both $-\lambda^2$ and $3$ are quadratic non-residue modulo $7$, by Hensel's Lemma it is easy to see that
$14n^2-\lambda^2$ can be represented by $R(x,y,z)$ over $\Z_7$. Moreover, as $14n-\lambda^2\equiv5\pmod8$, by Hensel's Lemma it is clear that $14n^2-\lambda^2$ can be represented by $R(x,y,z)$ over $\Z_2$. Hence $14n^2-\lambda^2$ can be represented by $\gen(R)$ if $n\ge\lambda/\sqrt{14}$. Furthermore, by \cite{JKS} we know that
$R(x,y,z)$ is a regular form. Hence $14n^2-\lambda^2$ can be represented by $R(x,y,z)$.
Set $h(x,y,z)=x^2+14y^2+42z^2$. We have the identity:
\begin{equation*}
h(3x+y,y,z)=3R(x,y,z).
\end{equation*}
Thus there exist $a,b,w\in\Z$ with $2\nmid a$ such that $42n^2-3\lambda^2=a^2+14b^2+42w^2.$ As
$42n^2-3\lambda^2\equiv-3\lambda^2\equiv a^2\pmod7$, we have $(a+2\lambda)(a-2\lambda)\equiv0\pmod7$. Without loss of generality, we assume that $a\equiv 2\lambda\pmod7$ (otherwise we replace $a$ by $-a$). Since
$42n^2-3\lambda^2\equiv0\equiv a^2-b^2\pmod3$, by changing the sign of $b$ if necessary, we assume that
$b\equiv a\pmod3$.

Since $a\equiv 2\lambda\pmod7$ and $2\nmid a$, there is an $r\in\{0,1,2\}$ such that
$a\equiv -5\lambda-14r\pmod{14}$ and $a\equiv -5\lambda-14r\pmod{3}$. By the Chinese Remainder Theorem, we may set $a=42s-5\lambda-14r$ with $s\in\Z$. Since $b\equiv a\pmod3$, there exists an integer $t$
such that $b=3t+\lambda-5r$. And we have
\begin{equation*}
42n^2-3\lambda^2=(42s-5\lambda-14r)^2+14(3t+\lambda-5r)^2+42w^2.
\end{equation*}
This implies that
\begin{equation}\label{equation E}
n^2=(\lambda-5s+t)^2+(-3r+4s+t)^2+(2r-s-t)^2+w^2.
\end{equation}
Let $x=\lambda-5s+t,\ y=-3r+4s+t,\ z=2r-s-t$. Then the desired result follows from (\ref{equation E}).

In view of the above, the proof of Theorem \ref{Thm 1.2} is now complete.\qed

\medskip
\noindent{\it Proof of Theorem \ref{Thm 1.3}}. Let $R^*(x,y,z)=3x^2+5y^2+7z^2+2xy$. Since the discriminant of
$R^*$ is $98$, when the prime $p\not\in\{2,7\}$,
$R^*(x,y,z)$ is unimodular in $\Z_p$ and hence it can represent all $p$-adic integers over $\Z_p$.
When $p=7$, $R^*(x,y,z)$ is $\Z_7$-equivalent to the form $3x^2-7y^2+7z^2$. Since $7\nmid \lambda$ and $-4\lambda^2$ and $3$ are quadratic non-residuess modulo $7$, by Hensel's Lemma, it is easy to see that
$14n-4\lambda^2$ can be represented by $R^*(x,y,z)$ over $\Z_7$. When $p=2$, we set
$$E_2(R^*)=\{m\in\Z_2: m\ne R^*(x,y,z)\ \text{for all}\ x,y,z\in\Z_2\}.$$
Since $R^*(x,y,z)$ is $\Z_2$-equivalent to the form $3x^2+10y^2+7z^2$, by the effective method of D. W. Jones
\cite[pp. 186--187]{Jones}, we have
\begin{align}\label{Except set A}
\{16m+14: m\in\Z_2\}\subseteq E_2(R^*),
\end{align}
and
\begin{align}\label{Except set B}
\{8m+2,8m+4,16m+6: m\in\Z_2\}\cap E_2(R^*)=\emptyset.
\end{align}
Since $2n-2\lambda\equiv2\pmod4$, we have $14\cdot(2n)-4\lambda^2\equiv 14(2n-2\lambda)\equiv4\pmod8$ and
$2n+1-2\lambda\not\equiv1\pmod8$. We also have $14\cdot(2n+1)-4\lambda^2\equiv2\pmod4$ and
$14\cdot(2n+1)-4\lambda^2\equiv14\cdot(2n+1-2\lambda)\not\equiv 14\pmod{16}$. Thus by (\ref{Except set A}) and
(\ref{Except set B}),
when $n$ satisfies the conditions described in Theorem \ref{Thm 1.3}, $14(2n+\delta)-4\lambda^2$ can be represented
by $\gen(R^*)$. Moreover, by \cite{JKS} we know that $R^*(x,y,z)$ is a regular form. This implies that
$14(2n+\delta)-4\lambda^2$ can be represented by $R^*(x,y,z)$ whenever $n$ satisfies the conditions in Theorem \ref{Thm 1.3}.
Let $l(x,y,z)=x^2+14y^2+35z^2$. We have the identity:
\begin{equation}\label{equation F}
l(x+5y, x, z)=5R^*(x,y,z).
\end{equation}
By (\ref{equation F}) for each $n\in\N$ satisfying the conditions in Theorem \ref{Thm 1.3},
there exist $a,b,c\in\Z$ such that $70(2n+\delta)-20\lambda^2=a^2+14b^2+35c^2$. Clearly, $a\equiv c\pmod 2$.
Since $70(2n+\delta)-20\lambda^2\equiv \lambda^2\equiv a^2\pmod7$, without loss of generality, we assume that
$a\equiv\lambda\pmod7$ (otherwise we may replace $a$ by $-a$). On the other hand, as
$70(2n+\delta)-20\lambda^2\equiv 0\equiv a^2-b^2\pmod5$, by changing the sign of $b$ if necessary, we may assume that
$b\equiv a\pmod5$. Thus we can find an $r\in\{0,1,2,...,9\}$ such that $a\equiv -\lambda-7r\pmod2$ and
$a\equiv -\lambda-7r\pmod5$. Then by the Chinese Remainder Theorem, we can write $a=70s-\lambda-7r$ with $s\in\Z$.
Since $b\equiv a\pmod 5$, we may set $b=5t-\lambda-2r$ with $t\in\Z$. As $c\equiv a\pmod2$, there is an integer
$u$ such that $c=2u-\lambda+r$. Hence we have
\begin{equation*}
70(2n+\delta)-20\lambda^2=(70s-\lambda-7r)^2+14(5t-\lambda-2r)^2+35(2u-\lambda+r)^2.
\end{equation*}
This implies that
\begin{equation}\label{equation G}
2n+\delta=(\lambda-s-t-u)^2+(r-s-t+u)^2+(-r+6s+t)^2+2(-4s+u)^2.
\end{equation}
Let $x=\lambda-s-t-u,\ y=r-s-t+u,\ z=-r+6s+t,\ w=-4s+u$. By (\ref{equation G}) we have
$2n+\delta=x^2+y^2+z^2+2w^2$ and $x+y+z+w=\lambda$. This completes the proof.
\qed
\maketitle
\section{Proof of Theorem \ref{Thm 1.4}}
\setcounter{lemma}{0}
\setcounter{theorem}{0}
\setcounter{corollary}{0}
\setcounter{remark}{0}
\setcounter{equation}{0}
\setcounter{conjecture}{0}

In \cite{MS}, X. -Z. Meng and the second author generalized Cauchy's Lemma and
obtained the following results.
\begin{lemma}\label{generalized Cauchy's Lemma 1}
{\rm (i)} {\rm(\cite[Lemma 2.4]{MS})} Let $a$ and $b$ be positive integers satisfying $b^2<4a$ and
$3a<b^2+2b+4$. Suppose that $2\nmid ab$ or $\ord_2(a)=1$ $(\ord_2(a)$ denotes the $2$-adic order of $a)$ and $2\mid
b$. Then there exist $s,t,u,v\in\N$ such
that
$$a=s^2+t^2+u^2+v^2\ \text{and}\ b=s+t+u+v.$$
{\rm (ii)} {\rm(\cite[Lemma 3.3]{MS})} Let $a$ and $b$ be positive integers with $a\equiv b\pmod2$ satisfying
$b^2<6a$ and $5a<b^2+2b+6$. Suppose that $2\nmid a$ or $\ord_2(a)=2$, and that $3\mid a$ or $3\nmid b$. Then
there exist $s,t,u,v\in\N$ such that
$$a=s^2+t^2+2u^2+2v^2\ \text{and}\ b=s+t+2u+2v.$$
{\rm (iii)} {\rm(\cite[Lemma 4.1]{MS})} Let $a$ and $b$ be positive integers satisfying $b^2<6a$ and
$5a<b^2+2b+6$. Suppose that $a\equiv b\pmod2$, and that $a\equiv3\pmod9$ or $3\nmid b$. Then there exist
 $s,t,u,v\in\N$ such that
 $$a=s^2+t^2+u^2+3v^2\ \text{and}\ b=s+t+u+3v.$$
\end{lemma}

Recall the notations given in (\ref{bound A})-(\ref{bound C}). We have the following result.
\begin{lemma}\label{length of interval}
For any positive integers $n,j$, let
$$\mathcal{I}_{j,k}=\l(\l(-1+\sqrt{j(n-1})\r)^{1/k},\l(\sqrt{(j+1)n}\r)^{1/k}\r).$$
We have

{\rm (i)} If $n>a_k(j,l)$, then the length of $\mathcal{I}_{j,k}$ is greater than $l$.

{\rm (ii)} If $n>c_k(j,l)$, then $\mathcal{I}_{j,k}$ contains at least one prime.
\end{lemma}
\begin{proof}
Clearly the interval $\l((jn)^{1/2k}, ((j+1)n)^{1/2k}\r)\subseteq \mathcal{I}_{j,k}$. Note that
\begin{align*}
((j+1)n)^{1/2k}-(jn)^{1/2k}&=\frac{n}{\sum_{s=0}^{2k-1}((j+1)n)^{s/2k}(jn)^{(2k-1-s)/2k}}
\\&>\frac{n^{1/2k}}{2k(j+1)^{(2k-1)/2k}}.
\end{align*}
Hence by computation and Lemma \ref{prime}, we get the desired results.
\end{proof}
\medskip

\noindent{\it Proof of Theorem \ref{Thm 1.4}}.
{\rm (1)} It is easy to verify that each $x\in \mathcal{I}_{3,k}$ satisfies the conditions $x^{2k}<4n$ and
$x^{2k}+2x^k+4>3n$. If $n\equiv2\pmod4$ and
$n>a_k(3,2)$, then by Lemma \ref{length of interval}, we can find an even integer $2b$ in $\mathcal{I}_{3,k}number of spinor genus$.
Similarly, if $n$ is odd and $n>c_k(3,2)$, then we can find a prime $p\in \mathcal{I}_{3,k}$. Then {\rm (i)} of
Theorem \ref{Thm 1.4} follows from Lemma \ref{generalized Cauchy's Lemma 1} {\rm (i)}. {\rm (2)} Via computation, each $x\in \mathcal{I}_{5,k}$ satisfies
the conditions $x^{2k}<6n$ and
$5n<x^{2k}+2x^k+6$. If $4\mid n$ and $n>a_k(5,6)$ (resp. $a_k(5,4)$), then we can find an
integer $2b\in \mathcal{I}_{5,k}$ with $n\equiv 2b$ and $3\nmid 2b$. If
$2\nmid n$ and $n>c_k(5,6)$ (resp. $c_k(5,4)$), then we can find a prime $p\in \mathcal{I}_{5,k}$.
Then {\rm (ii)} and {\rm (iii)}
of Theorem \ref{Thm 1.4} follow from parts (ii) and (iii) of Lemma \ref{generalized Cauchy's Lemma 1} respectively.

In view of the above, we complete the proof.\qed

\maketitle
\section{Some open problems}
\setcounter{lemma}{0}
\setcounter{theorem}{0}
\setcounter{corollary}{0}
\setcounter{remark}{0}
\setcounter{equation}{0}
\setcounter{conjecture}{0}

Inspired by Sun's 1-3-5 conjecture and part (i) of Theorem \ref{Thm 1.4}, we pose the following conjecture.

\begin{conjecture}\label{Conjecture A}
For any $a,b,c,d\in\N$ with $a^2+b^2+c^2+d^2\ne0$,
there exists a set $\mathcal{S}\subseteq\N$ with density zero such that
any sufficiently large integer can be written as $x^2+y^2+z^2+w^2$
 $(x,y,z,w\in\Z)$ with $ax+by+cz+dw\in\mathcal{S}$.
\end{conjecture}

In contrast with the above, motivated by Theorems \ref{Thm 1.2}-\ref{Thm 1.3}, we pose the following conjecture.

\begin{conjecture}\label{Conjecture B}
For any $a,b,c,d\in\N$ with $a^2+b^2+c^2+d^2\ne0$,
there exists a finite set $\mathcal{S}\subseteq\N$ such that
any sufficiently large integer can be written as $x^2+y^2+z^2+2w^2$ $(x,y,z,w\in\Z)$
 with $ax+by+cz+dw\in\mathcal{S}$.
\end{conjecture}

\end{document}